\documentclass[12pt]{article}
\usepackage{graphicx}
\usepackage{xcolor}
\usepackage{amssymb,amsfonts,amsthm,amsmath,calligra,tikz}
\usepackage[utf8]{inputenc}
\usepackage{accents}
\usepackage{slashed}
\usepackage{yfonts}
\usepackage{mathrsfs,pifont}
\theoremstyle{definition}
\usepackage{tikz}
\usepackage[all]{xy}
\usepackage{enumitem}
\usepackage{float}

\def\be{\begin{eqnarray}}
\def\ee{\end{eqnarray}}

\def\matZ{{\mathbb{Z}}}
\def\matR{{\mathbb{R}}}
\def\matQ{{\mathbb{Q}}}
\def\matC{{\mathbb{C}}}

\newcommand{\Or}{\textsf{O}}
\def\mc{{\bs \nu}_{b}}

\newcommand{\Stab}{\mathrm{Stab}}
\newcommand{\bA}{\mathsf{A}}
\newcommand{\bZ}{\mathsf{Z}}

\newcommand{\bT}{\mathsf{T}}
\newcommand{\bK}{\mathsf{K}}

\newcommand{\rrr}{\color{black}}

\newcommand{\Lie}{\mathrm{Lie}}
\newcommand{\Pic}{\mathrm{Pic}}

\newcommand{\cB}{\mathscr{B}}
\newcommand{\cE}{\mathscr{E}}

\let\bs\boldsymbol

\def\ss{\hat{\bs{s}} }

\def\rr{\bs{r}}
\def\nn{\bs{n}}

\def\wall{\mathsf{w}}

\def\ind{\mathrm{ind}}

\theoremstyle{definition}
\newtheorem{Definition}{Definition}
\newtheorem{Proposition}{Proposition}
\newtheorem{Lemma}{Lemma}
\newtheorem{Corollary}{Corollary}
\newtheorem{Conjecture}{Conjecture}
\newtheorem{Theorem}{Theorem}

\newtheorem{Note}{Note}

\newcommand{\fC}{\mathfrak{C}} 

\makeatletter
\newcommand{\somespecialrotate}[3][]{%
\begingroup
\sbox\@tempboxa{#3}%
\@tempdima=.5\wd\@tempboxa
\sbox\@tempboxa{\rotatebox[#1]{#2}{\usebox\@tempboxa}}%
\advance\@tempdima by -.5\wd\@tempboxa
\mbox{\hskip\@tempdima\usebox\@tempboxa}%
\endgroup}
\linethickness{1 pt}

\begin{document}
\title{Pursuing quantum difference equations I:
stable envelopes of subvarieties}
\author{Yakov Kononov, Andrey Smirnov}
\date{}
\maketitle
\thispagestyle{empty}
	
\begin{abstract}
Let $X$ be a symplectic variety equipped with an action of a torus~$\bA$. Let $\mc \subset \bA$ be a finite cyclic subgroup. We show that K-theoretic stable envelope of the fixed point set $X^{\mc}\subset X$ can be obtained via a limit of the elliptic stable envelopes of $X$.  An example of $X$ given by the Hilbert scheme of points in the complex plane is considered in detail.   
\end{abstract}
	
\setcounter{tocdepth}{2}
	
\section{Introduction}	
\subsection{}
The development of the theory of elliptic stable envelope was initiated by M.~Aganagic and A.~Okounkov in \cite{AOElliptic}. Since then the theory has found  remarkable applications to various areas of mathematics. To list just a few, stable envelopes can be related to the so-called Bethe vectors in integrable models \cite{OkBethe}, they provide a new description of standard bases for quantum groups \cite{Hik}, they manifest themselves as weight functions for solutions of the qKZ equations \cite{varch,RTV,konno1,konno2}, they provide explicit formulas for the R-matrices of various algebras (Yangians, quantum loop algebras and elliptic quantum groups) \cite{MO,Neg,ZinByk,InstR}. Stable envelopes also find important applications in Donaldson-Thomas theory of threefolds \cite{KOO}, and combinatorics of symmetric polynomials~\cite{NegGor,Wen,NegPier,GenJacks,MS}. We note also that the elliptic stable envelopes appear naturally in physics of $3d$ $\mathcal{N}=4$ supersymmetric gauge theories \cite{DedushNekr,BZ2}, they also describe the monodromy of partitions functions, studied for instance in \cite{KorGa,Dim}.

Initially, the theory was built as a tool to describe the monodromy of qKZ-like equations and quantum difference equations associated with the quiver varieties \cite{OS}. These ideas were outlined in \cite{OkBethe,AOF} as a generalization  of earlier  developments, such as for instance \cite{TV1,TV2,TV3,EV,TolMon}.

In geometric {\rrr the} approach, qKZ equations and quantum difference equations describe $q$-holonomic modules generated by  {\it vertex functions} of symplectic varieties \cite{pcmilect}. These developments revealed a deep interaction between Gromov-Witten type enumerative theories and representation theory. We refer to \cite{Pushk1,Pushk2,KorZe} for recent progress in this direction, see also \cite{Dink1,Dink2} for the description of vertex functions in more specific situations.

The elliptic stable envelope relates the enumerative invariants of  symplectic varieties to enumerative invariants of the {\it symplectic dual} varieties \cite{AOElliptic}. This suggests that  stable envelopes provide a natural tool to work with symplectic duality (or 3d-mirror symmetry). This idea was first emphasised by A. Okounkov in his talk ``{\it Enumerative symplectic duality}'' during the 2018 MSRI  workshop ``Structures in Enumerative Geometry'' (the talk is accessible from MSRI web-page) and further examined in several special cases in~\cite{RSVZ1,MirSym2, SZ}.

\subsection{} 
An interesting problem in {\rrr the} enumerative geometry of symplectic varieties is to find a better description of the corresponding $q$-difference equations.  Even though this problem has been  partly addressed in \cite{pcmilect,OS},  the treatment developed there is not entirely geometric and  relies on the techniques of Hopf algebras invented earlier in \cite{EV}. 

The analysis of the {\it monodromy} of these equations leads to a new geometric approach, which describes the building blocks of the $q$-difference equations (for instance the dynamical wall-crossing operators, see Section 5.3.1 in \cite{OS}) by special limits of the elliptic stable envelopes. This paper was mainly motivated by this idea and we consider it as a first natural step in this research direction.  Here we study special limits of the elliptic stable envelopes which arise in the following way: let $\bA$ be a torus acting on a symplectic variety $X$ by automorphisms, let $\mc \subset \bA$ be a cyclic subgroup of finite order. The inclusion $X^{\mc} \subset X$ induces a morphism of the elliptic cohomology schemes
$i: \textrm{Ell}_{\bT}(X^{\mc})\to \textrm{Ell}_{\bT}(X)$.
In this setup, the elliptic cohomology scheme of the $\mc$-fixed subset admits certain transformations $\omega_{\wall}: \textrm{Ell}_{\bT}(X^{\mc})\to \textrm{Ell}_{\bT}(X^{\mc})$ which preserve its structure. These transformations act by shifting the equivariant parameters $\omega_{\wall}: a \to a q^{\wall}$ by special elements  $\wall \in \Lie_{\matQ}(\bA)$ ($q$ denotes the modular parameter of the underlying elliptic curve $E$).

In Theorem \ref{manth} we prove that the elliptic stable envelope of $X$ twisted by $\omega_{\wall}$ in the limit $q=0$ converges to the K-theoretic stable envelope of the $\mc$-fixed subvarieties. Schematically,
\begin{small}
\be \label{stblim}
\textrm{Elliptic stable envelope of $X$ } \stackrel{q\to 0,z\to 0_{\mathfrak{D}}}{\longrightarrow} \textrm{K-theoretic stable envelope of $X^{\mc}$}
\ee
\end{small}
where $z\to 0_{\mathfrak{D}}$ denotes certain vanishing of K\"ahler parameters which controls the slope of the K-theoretic stable envelopes, see (\ref{limitdef}). 

In Section \ref{hssec} we apply this result to $X$ given by the Hilbert scheme of points in $\matC^2$. In this case the components of the fixed set $X^{\mc}$
are isomorphic to the Nakajima quiver varieties associated with cyclic quivers.    Theorem~\ref{hsthm} then establishes exact correspondence between
stable envelopes for these varieties. Our results here are related to the conjectures proposed in \cite{NegGor}, and we expect that their conceptual proofs will be obtained along these lines.

In the last section we consider a special case of $\mc$ given by a subgroup of framing torus of a Nakajima quiver variety $X$.  In this situation the twists $\omega_{\wall}$ and fixed sets $X^{\mc}$ are labeled by certain arrangement of hyperplanes in $\Lie_{\matR}(\bA)$.  The K-theoretic stable envelopes of $X^{\mc}$ arising  in the limit (\ref{stblim}) for all choices of $\omega_{\wall}$ are described by Theorem \ref{manth2}.

\section*{Acknowledgements}
We thank our teacher Andrei Okounkov for drawing our attention to the problem. We also thank Boris Feigin, Henry Liu and Zijun Zhou for useful discussions. 
This work was initiated during the AMS Mathematics Research Community meeting on Geometric Representation Theory and Equivariant Elliptic Cohomology at Rhode Island in June 2019 and the workshop ``Elliptic cohomology days'' at the University of Illinois, Urbana-Champaign. The authors are indebted to the organizers and all participants for very fruitful discussions and creative scientific atmosphere. 

The work of A. Smirnov is supported by NSF grant DMS - 2054527, by the Russian Science Foundation under grant 19-11-00062 and is performed in Steklov Mathematical Institute of Russian Academy of Sciences.

\section*{Data availability}
Data sharing not applicable to this article as no datasets were generated or analysed during the current study.

\section{Elliptic stable envelopes}

\subsection{} 
Let $X$ be a symplectic variety with an action of algebraic torus $\bT$. As usual, we assume that the action of $\bT$ scales the symplectic form with a character which we denote $\hbar^{-1}$. We denote by $\bA:=\ker(\hbar^{-1})\subset \bT$ the codimension one subtorus preserving the symplectic form.   

We assume that $X^{\bT}$ is finite. We assume also that {\rrr the} elliptic stable envelope exists for $X$. It is well known that the class of symplectic varieties satisfying this condition is quite large. For example, it includes all Nakajima quiver varieties, see Theorem 3 in \cite{AOElliptic}. 

For the definition of the elliptic stable envelope and basics of elliptic cohomology we refer to Sections 2-3 in \cite{AOElliptic} and Section 2 in \cite{EllipticHilbert}, in particular, Subsection 2.13 in \cite{EllipticHilbert} deals with the case of finite $X^{\bT}$.

\subsection{}
Let $\lambda \in X^{\bT}$ be a fixed point. By our assumption, for any choice of a chamber $\fC \subset \Lie_{\matR}(\bA)$ and a polarization $P\in K_{\bT}(X)$ (for definitions see, for instance, Section 2 of \cite{EllipticHilbert}) we have well defined elliptic stable envelope $\Stab^{Ell}_{X,\fC,P}(\lambda)$. By definition, $\Stab^{Ell}_{X,\fC,P}(\lambda)$ is a section of a certain bundle (Section 2.13 in \cite{EllipticHilbert}) over the extended elliptic cohomology scheme
$$
\mathsf{E}_{\bT}(X)=\textrm{Ell}_{\bT}(X)\times \cE_{\Pic(X)} 
$$
where $\textrm{Ell}_{\bT}(X)$ denotes the $\bT$-equivariant elliptic cohomology scheme of $X$ and $\cE_{\Pic({X})}=E\otimes_{\matZ} \Pic(X)$ 
for a family of elliptic curves $E=\matC^{\times}/q^{\matZ}$ over the punctured disc $0<|q|<1$. 

Recall that $\mathsf{E}_{\bT}(X)$ is a scheme over the extended elliptic cohomology scheme of a point:
$$
\mathsf{E}_{\bT}(X)\stackrel{\pi}{\longrightarrow}\cB_{\bT,X}:=\textrm{Ell}_{\bT}(pt)\times   \cE_{\Pic(X)}\cong E^{\dim(\bT)+\textrm{rk}(\Pic(X))}. 
$$
The coordinates on the abelian variety $\textrm{Ell}_{\bT}(pt)$ are usually called the equivariant parameters. We denote them by $a$ (for those corresponding to $\bA$) and $\hbar$ (cosponsoring to $\bT/\bA$). The coordinates in $\cE_{\Pic(X)}$ are referred to as the K\"ahler parameters and are denoted by $z$. 

\subsection{} 
We recall that the elliptic cohomology scheme has the following structure:
$$
\mathsf{E}_{\bT}(X)=\left(\coprod\limits_{\lambda\in X^{\bT}} \widehat{\Or}_{\lambda}\right)/ \Delta 
$$
where $\widehat{\Or}_{\lambda}\cong \cB_{\bT,X}$ and $\Delta$ denotes the data describing how the fixed point components $\widehat{\Or}_{\lambda}$ glue  to form $\mathsf{E}_{\bT}(X)$, see Section 2 in \cite{RSVZ1}. We  denote the restriction of the elliptic stable envelope to the fixed point components by
$$
T_{\lambda,\mu}(a,z)=\left.\Stab^{Ell}_{X,\fC,P}(\lambda)\right|_{\widehat{\Or}_{\mu}}. 
$$
The components $T_{\lambda,\mu}(a,b)$ represent sections of certain line bundles over the abelian varieties $\widehat{\Or}_{\mu}$ and thus can be expressed in terms of the odd Jacobi theta function associated with $E$:
\be \label{prodform}
\vartheta(x)=(x^{1/2}-x^{-1/2})\prod\limits_{i=1}^{\infty} (1-x q^{i})(1-q^{i}/x).
\ee 
Note that in the multiplicative notations odd means
\be \label{oddthe}
\vartheta(1/x)=-\vartheta(x).
\ee
The quasiperiods of these sections are governed by 
\be \label{thetatransf}
\vartheta(x q)=-\dfrac{1}{x\sqrt{q}} \vartheta(x).
\ee
Given a K-theory class $P$ we will denote by $\Theta(P)$ the corresponding elliptic Euler class. For example, if the non-constant part of $P$ is of the form:
$$
\sum_i a_i - \sum_j b_j \in K_{\bA}(pt)
$$
where $a_i,b_i$ are some non-trivial characters of $\bA$ then, explicitly
\be \label{thdef}
\Theta(P)=\dfrac{\prod\limits_{i} \vartheta(a_i)}{\prod\limits_{j} \vartheta(b_j)}.
\ee
We note that the $q=0$ limit equals:
$
\left.\vartheta(x)\right|_{q=0}=x^{1/2}-x^{-1/2}.
$

\subsection{}
By definition of the elliptic stable envelope, the sections $T_{\lambda,\mu}(a,z)$ are {\it holomorphic} in equivariant parameters $a$. 
The important feature of sections $T_{\lambda,\mu}(a,z)$ is that they are also {\it balanced} in a {\rrr suitable normalization \eqref{balsec}}.

Let
$\cE=E^{n}\times E^{m}$ be an abelian variety. We denote the coordinates   on the factors by $a=(a_1,\dots, a_n)$ and $z=(z_1,\dots,z_m)$.  Let $s(a,z)$ be a  meromorphic section of degree zero line bundle over $\cE$. 

\begin{Definition}
	{\it We say that $s(a,z)$ is balanced in the variables $a$ if in coordinates it can be represented in the following form:
		\be \label{balform}
		s(a,z)=\sum \, \prod\limits_{l} \dfrac{\vartheta(a^{l}\dots)}{\vartheta(a^l\dots)}
		\ee	
		where $a^l=a_1^{l_1}\dots a_n^{l_n}$ denote monomials in the variables $a$ and $\dots$ stands for monomials in the rest of variables $z$. }
\end{Definition}

For example, the following section over $E\times E$:
$$
s(a,z)={\frac {\vartheta \left( a z \right) }{\vartheta \left( a \right) \vartheta
		\left( z \right) }}+{\frac {\vartheta \left( {a}^{2} z \right) \vartheta
		\left( a \right) }{\vartheta \left( {a}^{2} \right) \vartheta \left( a z
		\right) }}
$$
is balanced in the variable $a$. It is also balanced in the variable $z$. 
But it is not balanced in variables $(a,z)$.
\subsection{} 
The elliptic functions of the form (\ref{balform}) have good behavior in the limit $q\to 0$:

\begin{Lemma} \label{wlim}
	{\it For any $\wall=(\wall_1,\dots,\wall_n) \in \matR^{n}$ and a section $s(a,z)$ balanced in variables $a$ the following limit exists
		\be \label{exlims}
		\sqrt{z}\lim\limits_{q\to 0} s(a q^{\wall}
		,z) \in \matQ(a,z)
		\ee
		where $a q^{\wall}=(a_1 q^{\wall_1},\dots,a_n q^{\wall_n})$, $\sqrt{z}$ denotes the square root of some monomial in variables $z_1,\dots,z_m$ and $\matQ(a,z)$ denotes the ring of rational functions.} 	
\end{Lemma}
\begin{proof}
	Let $\wall\in \matR$. 	The Lemma follows immediately from the following identity
	\be \label{thetlim}
	\lim\limits_{q\to 0} \dfrac{\vartheta(z a  q^{\textsf{w}})}{\vartheta(a q^{\textsf{w}})} = \left\{\begin{array}{ll} 
		z^{-\lfloor\textsf{w}\rfloor - 1/2 }, & \textsf{w} \not\in \matZ  \\
		\dfrac{1- a z}{1-a} z^{-\textsf{w}-1/2}, & \textsf{w} \in \matZ
	\end{array}	\right.
	\ee
	where $\lfloor \wall \rfloor$ stands for the integral part of $\wall$. This identity, in turn, can be derived from (\ref{prodform}).	
\end{proof}

Natural examples of balanced sections are provided by restrictions of the elliptic stable envelopes to the components of the fixed points. For $\lambda,\mu \in X^{\bT}$ let us consider the following section
\be
\label{balsec}
s(a,z)=\dfrac{T_{\lambda,\mu}(a,z)}{\Theta(P_\mu)}=\dfrac{\left.\Stab^{Ell}_{X,\fC,P}(\lambda)\right|_{\mu}}{\Theta(P_{\mu})}.
\ee
Here $a$ and $z$ denote the equivariant and K\"ahler parameters, which are the coordinates on abelian variety $\widehat{\Or}_{\mu}$ and  $P_{\mu}$ denotes the restriction of the polarization $P$ the fixed point $\mu$.  

\begin{Proposition} 
	{\it If $X$ is a hypertoric variety then (\ref{balsec})
		
		\noindent 1) is balanced in the equivariant parameters $a$,
		
		\noindent 2) is balanced in the K\"ahler parameters $z$.
	
	    \noindent 3) has poles separately in $a$ and $z$
}  
\end{Proposition}
The property 3) means that for $s(a,z)$ it is allowed to have factors $\vartheta(a)\vartheta(z)$
but not $\vartheta(az)$ in denominators of (\ref{balform}).
\begin{proof}
	For the hypertoric varieties, the formulas for the elliptic stable envelopes of fixed points can be described very explicitly as certain products of theta functions, see Section 4.1.3 of \cite{AOElliptic} or Section 3.2 in \cite{SZ}.  These hypertoric formulas are explicitly balanced separately in equivariant and K\"ahler parameters, and have separated poles.
\end{proof}

\begin{Corollary} \label{balcor}
	{\it If $X$ is a quiver variety with finite $X^{\bT}$ then (\ref{balsec}) has properties 1), 2), 3).}  	
\end{Corollary}
\begin{proof}
 For quiver varieties,  the elliptic stable envelope of a fixed point $\lambda \in X^{\bT}$ can be expressed in terms of the elliptic stable envelopes of fixed points in the hypertoric variety given by the  {\it abelianization} of $X$. We refer to Section 4 of \cite{AOElliptic} (in particular Section 4.3) where the details of the abelianization procedure are explained.  
\end{proof}

We expect that these properties of the elliptic stable envelope hold in general.

\begin{Conjecture} \label{conj1}
	{\it Let $X$ be a smooth symplectic variety with finite $X^{\bT}$ for which the elliptic stable envelope exists. Then 1),2),3) hold for (\ref{balsec}). }
\end{Conjecture}

It has been shown that the elliptic stable envelopes exists for quite general examples of $X$, we refer to
\cite{okounkov2021inductive,okounkov2021nonabelian} for discussion.

\subsection{} 
From the proof of Lemma \ref{wlim} it is clear that for generic $\textsf{w}$ the limit (\ref{exlims}) does not depend on  variables $a$, i.e, is an element of $\matC(z)$. The points for which this is not true play crucial role.

\begin{Definition} \label{resdef}
{\it Let $s(a,z)$ be a section balanced in variables $a$.
The point $\textsf{w}\in \matR^{n}$ is called a \underline{resonance} of $s(a,z)$ if the limit (\ref{exlims}) is a non-trivial function of $a$: 
$$
\sqrt{z}\lim\limits_{q\to 0} s(a q^{\wall}
		,z) \not\in \matC(z)
$$
We say that $\textsf{w}$ is a resonance of a collection of balanced sections $\{s_i(a,z)\}_{i\in I}$ if
it is a resonance for at least one of them.
}
\label{resdef}
\end{Definition}

We will denote by $\textsf{Res}(\{s_i(a,z)\}_{i\in I})\subset \matR^n$ the set of resonances of a collection $\{s_i(a,z)\}_{i\in I}$.

Assume we are given a finite set of $a$-{\rrr balanced} sections $\{s_i(a,z)\}_{i\in I}$. Consider the set of weights
$$
L=\{l=(l_1,\dots,l_n) \in \matZ^{n}: \textrm{appearing in (\ref{balform}) for all}  \ \ s_{i}(a,z), i\in I\}
$$
Let $\Lambda^{*} \subset \matR^{n}$ be the lattice generated by $L$, and let $\Lambda$ be the dual lattice. We can assume that $\Lambda \subset \matR^{n}$ by identifying $\matR^{n}$ with its dual.

\begin{Proposition} \label{prores}
{\it The set $\mathsf{Res}(\{s_i(a,z)\}_{i\in I})$ is a $\Lambda$-periodic  arrangement of hyperplanes in $\matR^{n}$.}
\end{Proposition}
\begin{proof}
It is clear from the explicit form (\ref{balform}) and limit (\ref{exlims}) that $\mathsf{w}$ is a resonance only if
$\langle l, \mathsf{w} \rangle =m$
for some integral $m$.
This is a $\Lambda$-periodic arrangement of hyperplanes.
\end{proof}

\begin{Note} \label{notelam}
$\mathsf{Res}(\{s_i(a,z)\}_{i\in I})$ is a subarrangement of the hyperplane arrangement given  by
$$
\langle l, \mathsf{w} \rangle =m
$$
for $l\in L$, $m\in \matZ$ (but does not necessarily {\rrr coincide} with it). 
\end{Note}

\subsection{}
The K-theoretic stable envelope (we refer to Section 9 of \cite{pcmilect} for its definition) can be obtained from the elliptic as the following limit:
\begin{Proposition}[Proposition 4.3 in \cite{AOElliptic}]
{\it For generic  $s\in \Pic(X)\otimes \matR$ we have:
\be \label{kthlim}
\det(P)^{-1/2}\otimes \lim\limits_{q\to 0} \left(\left.\Stab^{Ell}_{X,\fC,P}(\lambda)\right|_{z=q^s}\right) \otimes \det(P_{\lambda,0})^{1/2}  = \Stab^{Kth,[s]}_{X,\fC,P}(\lambda)
\ee
where $\Stab^{Kth,[s]}_{X,\fC,P}(\lambda)$ is the K-theoretic stable envelope of $\lambda$ with slope  $s$. $P_\lambda$ denotes the restriction of $P$ to a fixed point $\lambda$ and $P_{\lambda,0}$ is the component of $P_{\lambda}$ which has zero degree in $a${\rrr,} see (\ref{polparts}).} 
\end{Proposition}

The $K$-theoretic stable envelopes for the slopes which are close to $0\in \Pic(X) \otimes \matR$ play a special role in representation theory, see Theorem~10.2.11 in~\cite{pcmilect} for an example. If $\mathscr{U}_{0} \subset \Pic(X) \otimes \matR$ is a small analytic neighborhood of $0$ then the K-theoretic stable envelope changes only when the slope crosses certain hyperplanes passing through $0\in \Pic(X) \otimes \matR$. These hyperplanes divide the neighborhood into a set of  chambers:
\be \label{dcham}
\mathscr{U}_{0} \setminus \{ \textrm{hyperplanes} \} = \coprod_i \mathfrak{D}_i.
\ee 
We will denote K-theoretic stable envelopes with the slope from these chambers by:
$$
\Stab^{\mathfrak{D}}_{X,\fC,P}(\lambda):=\Stab^{Kth,[s]}_{X,\fC,P}(\lambda), \ \ \ s\in \mathfrak{D}.
$$
If we denote
\be \label{limitdef}
\lim\limits_{z\to 0_{\mathfrak{D}}} f(z):= \lim\limits_{q\to 0} f(q^{s}), \ \ s\in \mathfrak{D}, 
\ee
then for small slopes  (i.e., from $\mathscr{U}_{0}$) the above proposition gives:
\begin{Proposition} \label{kthcorol}
{\it Let us Denote by 
$$ \label{zlim}
S(\lambda):= \det(P)^{-1/2}\otimes \left( \lim\limits_{q\to 0} \Stab^{Ell}_{X,\fC,P}(\lambda)\right)  \otimes \det(P_{\lambda,0})^{1/2} \in K_{\bT}(X)\otimes \matC(z)
$$	
then
\be \label{zlim2}
\lim\limits_{z\to 0_{\mathfrak{D}}} S(\lambda) = \Stab^{\mathfrak{D}}_{X,\fC,P}(\lambda)
\ee	 } 
\end{Proposition}

\subsection{}
From the definition of the elliptic stable envelope we know that the section~(\ref{balsec}) has the following quasiperiods: 
\be \label{quasper}
s(a q^{\sigma},z)=z^{\chi_{\lambda}(\sigma,\cdot)-\chi_{\mu}(\sigma,\cdot)}  s(a,z) \ \ \ s(a ,zq^{\delta})=  a^{\chi_{\lambda}(\cdot,\delta)-\chi_{\mu}(\cdot,\delta)}  s(a,z) 
\ee
where  $\sigma \in \textrm{cochar}(\bA)$  and  $\delta \in \textrm{cochar}(\bK)$, and $\chi_{\lambda}$ is the pairing defined in Section 2.1.7 of \cite{OS}:
$$
\chi_{\lambda}: \textrm{cochar}(\bA)\times \textrm{cochar}(\bK) \to \matZ.
$$
(here we denote the K\"ahler torus of $X$ by $\bK=\Pic(X)\otimes_{\matZ}\matC^{\times}$).
By Lemma~\ref{wlim} this section has a well defined limit when $q\to 0$, moreover:
\begin{Lemma} \label{limexlem}
	{\it 
	If $\wall \in \mathrm{cochar}(\bA)\otimes \matR$ and $\textsf{v}\in \mathrm{cochar}(\bK) \otimes \matR$ then	the limits
		$$
		\lim_{z\to0_{\mathfrak{D}}}\Big(z^{-\chi_{\lambda}(\wall,\cdot)+\chi_{\mu}(\wall,\cdot)} \lim\limits_{q\to 0} s(a q^{\wall},z)\Big) 
		\ \ \ 
		\text{and} \ \ \ 
		\lim_{a\to0_{\mathfrak{C}}}\Big(a^{-\chi_{\lambda}(\cdot,\textsf{v})+\chi_{\mu}(\cdot,\textsf{v})} \lim\limits_{q\to 0} s(a ,z q^{\textsf{v}})\Big)
		$$
		exist for all chambers $\mathfrak{C}$ and $\mathfrak{D}$.} 
\end{Lemma}
\begin{proof}
Assume that both $\bA$ and $\bK$ are one-dimensional. The general case then follows from choosing arbitrary one-dimensional subtori in $\bA$ and $\bK$. We prove the Lemma for the first limit. For the second the argument is the same after switching the roles of $a$ and $z$.

{\rrr For a one-dimensional torus $\bK$ there are only two chambers. Thus, to show that the limits $z\to0_{\mathfrak{D}}$ exist for all chambers $\mathfrak{D}$
we need to show that the expression
$$
z^{-\chi_{\lambda}(\wall,\cdot)+\chi_{\mu}(\wall,\cdot)} \lim\limits_{q\to 0} s(a q^{\wall},z)
$$ 
has well defined limits as $z\to 0$ and $z^{-1}\to 0$.

As $s(a,z)$ is balanced and poles in $a$ and $z$ are separated, it must have the form:
	\be \label{terms}
	s(a,z) = \sum \, f(a) g(z)\prod\limits_{i} \dfrac{\vartheta(a^{n_i} z^{m_i})}{\vartheta(a^{n_i}) \vartheta(z^{m_i})} 
	\ee
	where $f(a)$ and $g(z)$ are some balanced sections depending only on $a$ and~$z$ respectively.
We note that 
$$
\lim\limits_{q\to 0} \vartheta(x)=x^{1/2}-x^{-1/2}
$$
which is obvious from the definition (\ref{prodform}).
Using this and (\ref{thetlim}) we compute
	$$
	\lim\limits_{q\to 0} s( a q^{\wall},z) =  r(a,z)	\prod\limits_{i} \left\{
	\begin{array}{ll}
 \dfrac{z^{-\lfloor \wall n_i\rfloor  m_i }}{z^{m_i}-1},& \wall n_i \not\in \mathbb{Z} \\
 \dfrac{(1-a^{n_i} z^{m_i}) z^{-n_i m_i \wall_i}}{(1-a^{n_i})(z^{m_i}-1)}, & \wall n_i \in \mathbb{Z}
	\end{array}\right.
	$$
	with $r(a,z)$ such that the limits 
	$
	\lim\limits_{z^{\pm 1}\to 0} r(a,z) \in \matQ(a)
	$ 
	exist. 
	
If $\bA$ and $\bK$ are one-dimensional then  $\wall \in \matR$ and $\chi_{\lambda}(\cdot,\cdot) \in \matZ$. Using~(\ref{thetatransf}) from (\ref{quasper}) we compute that
	\be \label{lsum}
	\chi_{\lambda}(\cdot,\cdot)-\chi_{\mu}(\cdot,\cdot) =- \sum\limits_{i} n_i m_i.
	\ee
	The terms in the sum (\ref{terms}) are sections of the same line bundle, and thus the value (\ref{lsum}) must be the same for each term. Thus, we have	
	$$
z^{-\chi_{\lambda}(\wall,\cdot)+\chi_{\mu}(\wall,\cdot)}	\lim\limits_{q\to 0} s( a q^{\wall},z) =  r(a,z)	\prod\limits_{i} \left\{
	\begin{array}{ll}
 \dfrac{z^{m_i(\wall n_i-\lfloor \wall n_i\rfloor)}}{z^{m_i}-1},& \wall n_i \not\in \mathbb{Z} \\
 \dfrac{(1-a^{n_i} z^{m_i})}{(1-a^{n_i})(z^{m_i}-1)}, & \wall n_i \in \mathbb{Z}
	\end{array}\right.
	$$	
It is clear that 
$$
\lim_{z^{\pm 1} \to 0} \dfrac{(1-a^{n_i} z^{m_i})}{(1-a^{n_i})(z^{m_i}-1)} < \infty, \ \ \ 	\lim\limits_{z^{\pm}\to 0} \dfrac{z^{m_i(\wall n_i - \lfloor \wall n_i \rfloor )}}{z^{m_i}-1}<\infty
$$
The second pair of limits follows from 
	$
	0 \leq \wall n_i - \lfloor \wall n_i \rfloor < 1.
	$}
\end{proof}

\section{Subvarieties invariant under finite subgroups}
\subsection{\label{wshift}}
Let $\mc \subset \bA$  be a cyclic subgroup of order $b$ and let $X^{\mc}$ be its fixed set. The action of $\bA$ on $X^{\mc}$ factors through the map 
$
\psi: \bA\to \bA^{'} = \bA/\mc. 
$
We denote $\mathsf{Z}= \psi^{-1} (q^{\textrm{cochar}(\bA')})$.
The group $\mathsf{Z}$ acts on $\textsf{E}_{\bT}(X^{\mc})$ by translations in the equivariant parameters $a\to a q^{\wall}$.

We fix an element $q^{\wall} \in \mathsf{Z}$ such that 
$(q^{\wall})^{b}\in q^{\textrm{cochar}(\bA)\setminus \{0\}}$ but $(q^{\wall})^{m}\not \in q^{\textrm{cochar}(\bA)}$ for $0<m<b$.  We denote the corresponding translation of the elliptic cohomology scheme by~$\omega_{\wall}$:
\[
\xymatrix{
	\mathsf{E}_{\bT}(X^{\mc})\ar[r]^{\omega_{\wall} }\ar[d]^{\pi}&
	\mathsf{E}_{\bT}(X^{\mc})\ar[d]^{\pi}\\
	\cB_{\bT,X^{\mc}}\ar[r]^{a\to a q^{\wall}}&\cB_{\bT,X^{\mc}}. }
\]

\subsection{}
The restriction of the polarization to a fixed point has the following decomposition:
\be \label{polparts}
P_{\lambda}=P_{\lambda,>0}\oplus P_{\lambda,<0}\oplus P_{\lambda,0} \in K_{\bA}(pt)
\ee
where the three terms denote the parts whose $\bA$-characters take positive, negative or zero values at the chamber $\fC$.
The positive part is called {\it index} of the fixed point $\lambda$:
$$
\ind_{\lambda} = P_{\lambda,>0}  \in K_{\bA}(pt).
$$
Similarly we have a decomposition of the tangent spaces at the fixed points:
$$
T_{\lambda} X=N^{+}_{\lambda} \oplus N^{-}_{\lambda}.
$$
Assume $\ind_{\lambda}$ is of the form
$$
\ind_{\lambda}=\sum\limits_{\sigma } \, a^{\sigma}.
$$
i.e., the sum is over the set of $\bA$-characters appearing in $\ind_{\lambda}$.
Then, for $\wall \in \textrm{cochar}(\bA)\otimes \matR$ we denote
$$
\lfloor \ind_{\lambda} \cdot \wall \rfloor = \sum_{\sigma} \, \lfloor \langle\sigma, \wall\rangle \rfloor .
$$

\begin{Lemma} \label{lllem} {\it 
If $P^{\mc}_{\lambda},N^{-,\mc}_{\lambda},{\ind^{\mc}_{\lambda}}$ denote  
$\mc$-invariant parts of $P_{\lambda},N^{-}_{\lambda}$ and ${\ind_{\lambda}}$, then for $\wall$ as in Section \ref{wshift}  we have:}
\begin{small}
\begin{align}
     \label{lemlim}
		 \lim\limits_{q\to 0} \left(\left[\dfrac{\Theta(N^{-}_{\lambda})}{\Theta(P_{\lambda})}\right]_{a=a q^{\wall}}	\right) =
		{\rrr (-1)^{\mathrm{rk}(\ind_{\lambda})-\mathrm{rk}(\ind^{\mc}_{\lambda})} \dfrac{\hbar^{\lfloor\ind_{\lambda} \cdot \wall \rfloor+\mathrm{rk}(\ind_{\lambda})/2 } }{\det(\ind^{\mc}_{\lambda}) \det(P_{\lambda,0})^{1/2}}
		 \dfrac{\Lambda^{\!\bullet}(\bar{N}^{-,\mc}_{\lambda}) }{\Lambda^{\!\bullet}(\bar{P}^{\mc}_{\lambda})}}
\end{align}
\end{small}

\end{Lemma}

\begin{proof}
In our notations, the tangent space at a fixed point $\lambda$ equals:
$$
T_{\lambda} X= P_{\lambda}+\hbar\, {P}^{*}_{\lambda}
$$
Thus, its repelling part is
$$
N^{-}_{\lambda}=P_{\lambda,<0}+\hbar\, (P_{\lambda,>0})^{*}
$$
and thus
$$
\dfrac{\Theta(N^{-}_{\lambda})}{\Theta(P_{\lambda})}=\dfrac{\Theta(P_{\lambda,<0})\Theta(\hbar\, (P_{\lambda,>0})^{*})}{\Theta(P_{\lambda,<0}) \Theta(P_{\lambda,>0}) \Theta(P_{\lambda,0})}=(-1)^{\textrm{rk}(\ind_{\lambda})} \dfrac{\Theta(\hbar^{-1}\, P_{\lambda,>0})}{\Theta(P_{\lambda,>0})} \dfrac{1}{\Theta(P_{\lambda,0})}
$$
where the last equality is by (\ref{oddthe}).
In this form, the limit is easily computed from (\ref{thetlim}). The result follows after some simple algebra. 	


\end{proof}

\subsection{}
{\rrr Let us recall that the elliptic stable envelope is normalized by its restriction near the diagonal:
$$
\left.\Stab^{Ell}_{X,\fC,P}(\lambda)\right|_{\lambda}=(-1)^{\mathrm{rk}(\ind_{\lambda})} \Theta(N^{-}_{\lambda})
$$
see Section 3.3.5 in \cite{AOElliptic}. We recall also that the K-theoretic stable envelope can be obtained from the elliptic via limit (\ref{kthlim}). In particular, the $K$-theoretic stable envelope in this approach is normalized by its diagonal components,
\be \label{ksdia}
\left.\Stab^{Kth,[s]}_{X,\fC,P}(\lambda)\right|_{\lambda}=\dfrac{(-1)^{\mathrm{rk}(\ind_{\lambda})} \hbar^{\mathrm{rk}(\ind_{\lambda})/2}}{\det(\ind_{\lambda})}\, \Lambda^{\!\bullet} (\bar{N}^{-}_{\lambda})
\ee
which we also assume in this paper.
 }
 

\subsection{} 
The chamber $\fC$ and the $\mc$-invariant part of the polarization $P^{\mc} \in K_{\bT}(X^{\mc})$ define the elliptic and K-theoretic stable envelopes for $X^{\mc}$.  The inclusion $X^{\mc} \to X$ induces a map of extended elliptic cohomology schemes:
$$
i: \ \ \textsf{E}_{\bT}(X^{\mc})\to \textsf{E}_{\bT}(X).
$$
If $\mathfrak{D}$ is  a chamber from (\ref{dcham}) then we denote by $\mathfrak{D}'$ the corresponding chamber for $X^{\mu_b}$ defined by the property:
\be \label{chamdp}
\kappa (\mathfrak{D}) \subset \mathfrak{D}' 
\ee
where  $\kappa:\Pic(X)\otimes \matR\to \Pic(X^{\mc})\otimes \matR$ is the induced map.

For $\wall \in \textrm{cochar}(\bA)\otimes \matR$ let us define a $\textrm{char}(\bK)$-valued function on $X^{\bT}$ by
$$
\lambda \to \chi_{\lambda}(\wall,\cdot) \in \textrm{char}(\bK).
$$
Here is our main result.
\begin{Theorem} \label{manth}
{\it Assume that Conjecture \ref{conj1} holds for a variety $X$.  Let $\omega_{\wall}$ be the translation in equivariant parameters as in Section \ref{wshift}. Define 
$$
S:=\lambda \to \Lambda^{\!\bullet}(\bar{P}^{\mc}) \circ \omega^{*}_{\wall} \circ i^{*} \left( \dfrac{\Stab^{Ell}_{X,\fC,P} (\lambda) }{\Theta(P)}  \right) \circ \det(P_{0,\lambda})^{1/2}, 
$$
then
\begin{align} \nonumber
 &\lim\limits_{z\to 0_{\mathfrak{D}}} \left( z^{\chi(\wall,\cdot)-\chi_{\lambda}(\wall,\cdot)}  \lim\limits_{q\to 0}  S(\lambda) \right) 
	 =\gamma_{\lambda}\, \Stab^{\mathfrak{D}'}_{X^{\mc},\fC,P^{\mc}}(\lambda).  \label{mainlim} 
\end{align}
with coefficient $\gamma_{\lambda}=\hbar^{\lfloor \ind_{\lambda}\cdot \wall\rfloor +\mathrm{rk}(\ind_{\lambda})/2- \mathrm{rk}(\ind^{\mc}_{\lambda})/2}$.
}
\end{Theorem}
\begin{proof}
Let us assume that $\bA\cong \matC^{\times}$. If not, we can choose a cocharacter $\matC^{\times} \to \bA$ whose image contains $\mc$, then the shift $\omega^{*}_{\wall}$ does not affect the equivariant parameters in $\bA/\matC^{\times}$ and thus they do not change the limit $q\to 0$.

Let us denote 
$$
E_{\lambda,\mu}(a,z):= i^{*}\left.\left(\dfrac{\Stab^{Ell}_{X,\fC,P} (\lambda) }{\Theta(P)}\right)\right|_{\mu}. 
$$ 	
These are the fixed point components of certain meromorphic section over $\textsf{E}_{\bT}(X^{\mc})$. We have
	\be \label{ecomp}
	E_{\lambda,\mu}(a q^{\wall},z)=\left.\omega^{*}_{\wall} \circ i^{*} \left(\dfrac{\Stab^{Ell}_{X,\fC,P} (\lambda) }{\Theta(P)}\right)\right|_{\mu} 
	\ee
By Corollary \ref{balcor} the sections $E_{\lambda,\mu}(a,z)$ are balanced, i.e., are of the form (\ref{balform}).  By Lemma \ref{wlim} the limit 
$$
\lim\limits_{q\to 0} E_{\lambda,\mu}(a q^{\wall},z)
$$ 
exists. By (\ref{thetlim}), only the factors $\vartheta(a^n\dots)$ ( $\dots$ stands for monomials in $z$ and $\hbar$) in the numerator and denominator of $E_{\lambda,\mu}(a q^{\wall},z)$ with ${\rrr b\mid n}$ can contribute a nontrivial function of $a$ in this limit. The factors with ${\rrr b\nmid n}$ in the limit $q\to 0$ can only produce monomials in $z$ and $\hbar$. 
The factors $\Theta(P)$ with ${\rrr b\mid n}$ are exactly those from the invariant part $\Theta(P^{\mc})$. Thus, one can cancel all  poles in the limit of (\ref{ecomp}) by tensoring it with $\Lambda^{\!\bullet}(\bar{P}^{\mc})$. We conclude that
$$
K_{\lambda,\mu}(a,z)=\Lambda^{\!\bullet}(\bar{P}^{\mc}_{\mu}) \otimes \lim\limits_{q\to 0} E_{\lambda,\mu}(a q^{\wall},z) \otimes \det(P_{\lambda,0})^{1/2}
$$
are holomorphic in equivariant parameters $a$.

These are the fixed point components of a holomorphic (in $a$) function on $\textrm{Spec}(K_{\bT}(X^{\mc}))\otimes \bK$, which we denote by	
\be \label{kclass}
	K(\lambda):=\Lambda^{\!\bullet}(\bar{P}^{\mc}) \circ  \lim\limits_{q\to 0}  \omega^{*}_{\wall} \circ i^{*} \left( \dfrac{\Stab^{Ell}_{X,\fC,P} (\lambda) }{\Theta(P)}  \right) \circ \det(P_{\lambda,0})^{1/2} 
\ee
From the support condition for $\Stab^{Ell}_{X,\fC,P}(\lambda)$ (see Section 3.3.5 \cite{AOElliptic}) we find that $K(\lambda)$ is supported at:
	\be \label{support}
	\textrm{Supp}(K(\lambda)) \subset X^{\mc} \cap \textsf{Attr}^{f}_X(\lambda) =\textsf{Attr}^{f}_{X^{\mc}}(\lambda).
	\ee
By definition of the elliptic stable envelope $\left.\Stab^{Ell}_{X,\fC,P} (\lambda)\right|_{\lambda}={\rrr (-1) ^{\mathrm{rk}(\ind_{\lambda})}}\Theta(N^{-}_{X,\lambda})$. The factors in $\Theta(N^{-}_{X,\lambda})$ with ${\rrr b\mid n}$ are exactly those in $\Theta(N^{-}_{X^{\mc},\lambda})$.  From Lemma~\ref{lllem} 
	we find that the diagonal components of $K(\lambda)$ have the form:
	\be \label{nrm}
	\left.K(\lambda)\right|_{\lambda}= \hbar^{\lfloor \ind_{\lambda}\cdot \wall\rfloor +\mathrm{rk}(\ind_{\lambda})/2 } \dfrac{(-1)^{\mathrm{rk}(\ind^{\mc}_{\lambda})}}{\det(\ind^{\mc}_{\lambda})}  \Lambda^{\! \bullet}(\bar{N}^{- }_{X^{\mc},\lambda}).
	\ee
	
	
	The K-theoretic stable envelope is characterized by $a$-degree bound on its fixed point components, see Section 9.1.9 in \cite{pcmilect}. In particular, Proposition~\ref{kthcorol} implies that we have the following bounds:
\begin{equation} 	
\label{bound}
\begin{split}
  &	\deg_\bA\left( \lim\limits_{z\to 0_{\mathfrak{D}}}\Big(  \Lambda^{\!\bullet}(\bar{P}_{\mu})\otimes \lim\limits_{q\to0}\left.\dfrac{\Stab^{Ell}_{X,\fC,P}(\lambda)}{\Theta(P)}\right|_{\mu} \Big)\right)   \\
 &	\subset 	\deg_\bA\left( \lim\limits_{z\to 0_{\mathfrak{D}}}\Big(  \Lambda^{\!\bullet}(
	\bar{P}_{\mu})\otimes \lim\limits_{q\to0}\left.\dfrac{\Stab^{Ell}_{X,\fC,P}(\mu)}{\Theta(P)}\right|_{\mu} \Big)\otimes s_{\lambda} \otimes s_{\mu}^{-1}\right)
\end{split} 
\end{equation}
where $s_{\lambda}$ denotes the restriction of a line bundle $s\in \mathfrak{D}$ from (\ref{dcham}). 
 
If we consider the same limits with additional  shift $\omega^{*}_{\wall}$ as in (\ref{kclass}) the only the terms $\vartheta(a^n\dots)$ with ${\rrr b\mid n}$ contribute. Thus, taking the $\mc$-invariant part of (\ref{bound}) we obtain:
	\begin{align}  \label{window}
	& \deg_{\bA}( \lim\limits_{z\to 0_{\mathfrak{D}}} z^{\chi_{\lambda}(\wall,\cdot)-\chi_{\mu}(\wall,\cdot)} \left.K(\lambda)\right|_{\mu})\subset
	\deg_{\bA}(\lim\limits_{z\to 0_{\mathfrak{D}}}\left.K(\mu)\right|_{\mu} \otimes s_{\lambda} \otimes s_{\mu}^{-1})
	\end{align}
Note that the limits exist by Lemma \ref{limexlem}. By $\deg_{\bA}(f)$ we denote the Newton polytope of a Laurent polynomial $f$. Inclusion (\ref{window}) {\rrr denotes} the inclusion of the corresponding Newton polytopes.
Now, (\ref{support}), (\ref{nrm})  and (\ref{window}) say that the K-theory class
$$
\lim\limits_{z\to 0_{\mathfrak{D}}} z^{\chi_{\lambda}(\wall,\cdot)-\chi(\wall,\cdot)} K(\lambda)
$$
satisfies all three defining properties of the K-theoretic stable envelope with slope $s$, 
see Section 9 in \cite{pcmilect}.

Comparing (\ref{nrm}) with (\ref{ksdia}) we find that the normalization of this K-theoretic stable envelope differs from the one accepted in this paper by a factor

$$
\hbar^{\lfloor \ind_{\lambda}\cdot \wall\rfloor +\mathrm{rk}(\ind_{\lambda})/2- \mathrm{rk}(\ind^{\mc}_{\lambda})/2}
$$
The theorem follows from the uniqueness of the stable envelope in K-theory  see Proposition 9.2.2 in \cite{pcmilect}. 
\end{proof}

\subsection{} 
For practical computations, it might be more convenient to formulate the above theorem as follows. Let us consider the normalized matrix of restrictions:
\be \label{ttild}
\tilde{T}_{\lambda,\mu}(a,z):=\dfrac{\left.\Stab^{Ell}_{X,\fC,P}(\lambda)\right|_{\mu} }{\left.\Stab^{Ell}_{X,\fC,P}(\mu)\right|_{\mu} }. 
\ee
This is a triangular matrix with trivial diagonal $\tilde{T}_{\lambda,\lambda}(a,z)=1$ and other coefficients given by certain  elliptic functions.
Similarly we denote
\be \label{Kkth}
\tilde{K}_{\lambda,\mu}(a,\hbar):=\dfrac{\left.\Stab^{\mathfrak{D}'}_{X^{\mc},\fC,P^{\mc}}(\lambda)\right|_{\mu} }{\left.\Stab^{\mathfrak{D}'}_{X^{\mc},\fC,P^{\mc}}(\mu)\right|_{\mu}}
\ee
the matrix of K-theoretic stable envelopes of $X^{\mc}$ with a slope from $\mathfrak{D}'$ normalized in the  same fashion. 

\begin{Theorem}  \label{thm2}
{\it Assume that Conjecture \ref{conj1} holds for a variety $X$, then the matrix $\tilde{K}(a,\hbar)$ can be obtained from the matrix $\tilde{T}_{\lambda,\mu}(a,z)$ as the following limit:	
\be \label{seclim}
\lim\limits_{z\to 0_{\mathfrak{D}}} Z \Big(\lim\limits_{q\to 0} \tilde{T}(a q^{\wall},z)\Big) Z^{-1} =H \tilde{K}(a,\hbar) H^{-1}
\ee	
where $Z$ denotes the diagonal matrix
$$
Z:=\left.\mathrm{diag}(z^{\chi_{\lambda}(\wall,\cdot)})\right|_{\lambda \in X^{\bT}}
$$
and $H$ denotes the diagonal matrix with eigenvalues:
$$
H_{\lambda,\lambda}:=\hbar^{\lfloor \ind_{\lambda}\cdot \wall \rfloor+\textrm{rk}(\ind_{\lambda})/2-\textrm{rk}(\ind^{\mc}_{\lambda})/2}\, \det(P_{\lambda,0})^{-1/2}\, 
$$

}
\end{Theorem}

\begin{proof}
We note that
\be \label{npol}
N_{\lambda}^{-}=P_{\lambda,<0}+\hbar \bar{P}_{\lambda,>0} 
\ee
and thus by (\ref{polparts}) we have 
\be \label{zfrat}
    \dfrac{\left.\Stab^{Ell}_{X,\fC,P}(\lambda)\right|_{\lambda}}{\Theta(P_{\lambda})}=\dfrac{(-1)^{\textrm{rk}(\ind_{\lambda})} \Theta(N^{-}_{\lambda})}{\Theta(P_{\lambda})}=\dfrac{1}{\Theta(P_{\lambda,0})} \dfrac{\Theta(P_{\lambda,>0} \hbar^{-1})}{\Theta(P_{\lambda,>0})}.
\ee
We conclude that this ratio is a balanced function in the equivariant parameters~$a$. Dividing any balanced function by this ratio is clearly a balanced function again and thus all elliptic functions (\ref{ttild}) are balanced in $a$. By Lemma \ref{wlim} we conclude that the limits $q\to 0$ in (\ref{seclim}) 
are well defined for all~$\wall$. 

Conjugation by the diagonal matrix $Z$ gives:
$$
 \tilde{T}_{\lambda,\mu}(a,z) \to  z^{\chi_{\mu}(\wall,\cdot)-\chi_{\lambda}(\wall,\cdot)} \tilde{T}_{\lambda,\mu}(a,z),
$$
Thus, the existence of the limit $z\to 0_{\mathfrak{D}}$ follows from Theorem \ref{manth} (note that the ratio (\ref{zfrat}) does not depend on the K\"ahler parameters and thus can not affect  asymptotic behavior at $z\to 0_{\mathfrak{D}}$). 

Applying Theorem \ref{manth} we find:
\begin{align}
& \left(\lim\limits_{z\to 0_{\mathfrak{D}}} Z \Big(\lim\limits_{q\to 0} \tilde{T}(a q^{\wall},z,\hbar,q)\Big) Z^{-1}\right)_{\lambda,\mu} = & \nonumber \\  \nonumber
& \dfrac{\hbar^{\lfloor \ind_{\lambda}\cdot \wall \rfloor+\textrm{rk}(\ind_{\lambda})/2-\textrm{rk}(\ind^{\mc}_{\lambda})/2}\, \det(P_{\lambda,0})^{-1/2}\, \left.\Stab^{\mathfrak{D}'}_{X^{\mc},\fC,P^{\mc}}(\lambda)\right|_{\mu}}{ \hbar^{\lfloor \ind_{\mu}\cdot \wall \rfloor+\textrm{rk}(\ind_{\mu})/2-\textrm{rk}(\ind^{\mc}_{\mu})/2}\, \det(P_{\mu,0})^{-1/2}\,\left.\Stab^{\mathfrak{D}'}_{X^{\mc},\fC,P^{\mc}}(\mu)\right|_{\mu}}= \tilde{K}_{\lambda,\mu}(a,\hbar) H_{\lambda,\lambda}/H_{\mu,\mu}&
\end{align}
which finished the proof. 
\end{proof}

\subsection{} 
For $\wall \in \Lie_{\matR}(\bA)$ let us consider the cyclic subgroup ${\bs \nu}_{\wall} =\langle e^{2\pi i \wall} \rangle \subset \bA$. We denote
$$
\textsf{Res}(X)=\{ \wall \in \Lie_{\matR}(\bA): X^{{\bs \nu}_{\wall}}\neq X^{\bA} \}.
$$
We will call the points from $\textsf{Res}(X) \subset \Lie_{\matR}(\bA)$ {\it resonances}. We will now show that this terminology is in agreement with  Definition \ref{resdef}.

\begin{Proposition}{\it 
\begin{enumerate} 
\indent 
The sets 
\item $S_1=\mathsf{Res}(\{ \tilde{T}_{\lambda,\mu}(a,z) \}_{\lambda,\mu \in X^{\bT}}),$
 \item  $S_2=\mathsf{Res}(X),$ 
    \item $S_3=\{ \wall \in \Lie_{\matR}(\bA): \langle \alpha,\wall \rangle + m=0, \ \  m\in \matZ, \ \ \alpha \in \mathrm{char}_{\bA}(T_{\lambda} X), \  \lambda\in X^{\bA}  \},$
\end{enumerate}
 are equal.}
\end{Proposition}

\begin{proof}
Assume that $\wall \in S_2$, then $X^{{\bs \nu}_{\wall}}\neq X^{\bA}$.  $X^{{\bs \nu}_{\wall}}$ is a $\bA$-invariant subvariety of $X$, with the same set of $\bA$-fixed points. Let $\lambda$ be an $\bA$-fixed point {\rrr in a} nontrivial component of $X^{{\bs \nu}_{\wall}}$ (i.e., this component does not consists of a single point $\lambda$). Let $a^{\alpha}$ be a an $\bA$ weight from $T_{\lambda} X^{{\bs \nu}_{\wall}}$. It is invariant under ${\bs \nu}_{\wall}$, which means that $e^{2 \pi i \langle \alpha, \wall\rangle }=1$, or that $\wall \in S_3$. Thus we showed that $S_2\subset S_3$.

Next assume $\wall \in S_3$. Then there exists a fixed point $\lambda$ and a direction in $T_{\lambda} X$ with character $\alpha$ for which $\langle \alpha,\wall \rangle \in \matZ$. This means that this whole direction in $X$ is preserved under the action of ${\bs \nu}_{\wall}$, i.e., $X^{{\bs \nu}_{\wall}}$ is larger that $X^{\bA}$. Thus $\wall \in S_2$ and therefore $S_3\subset S_2$. We conclude $S_2=S_3$. 

Next, assume that $\wall \in S_2$. The variety $X^{{\bs \nu}_{\wall}}$ is a non-trivial (not finite) and thus the matrix of restrictions of K-theoretic stable envelopes $\tilde{K}_{\lambda,\mu}(a)$ defined by (\ref{Kkth}) depends on parameters $a$ non-trivially. 
By Theorem \ref{thm2}, 
$$
\lim\limits_{z\to 0_{\mathfrak{D}}} Z \Big(\lim\limits_{q\to 0} \tilde{T}(a q^{\wall},z)\Big) Z^{-1}
$$
is then a non-trivial function of $a$. This is only possible if $\lim\limits_{q\to 0} \tilde{T}(a q^{\wall},z)$ is a non-trivial function of $a$. Thus,  thus $\wall \in S_1$ and so $S_2\subset S_1$.

Finally, the $a$-{\rrr balanced} sections $\tilde{T}_{\lambda,\mu}(a,z)$ as defined by (\ref{ttild}) all have denominators 
$$
\left.\Stab^{Ell}_{X,\fC,P}(\lambda)\right|_{\lambda} =\pm \prod\limits_{{a^{\alpha} \in \textrm{weights}_{\bA}(T_{\lambda}X)}\atop \langle l, \sigma \rangle <0} \, \vartheta(a^{l}\dots)
$$
where $\dots$ stand for some power of $\hbar$. We conclude that the set $S_1$ is a subset of the hyperplane arrangement 
$$
\langle l,\wall \rangle =m,
$$
where $m\in \matZ$ and $l$ runs over all characters appearing  $T_{\lambda} X$ such that $\langle l, \sigma \rangle <0$ (by Proposition \ref{prores} and Note \ref{notelam}). But, this is clearly the same set as for all $l$ appearing in $T_{\lambda} X$.
In other words $S_1 \subset S_3$. 

\end{proof}

\section{Application to the case of the Hilbert Scheme \label{hssec}}
In this section we consider an application of Theorem \ref{manth} to the case of $X$ given by the Hilbert scheme of $n$ points on $\matC^{2}$. In this section we follow notations of \cite{EllipticHilbert}, where the explicit formula  for the elliptic stable envelope for this variety was obtained. 
In particular, the tori $\bA\subset\bT$ acting on $X$, the set of fixed points $X^{\bT}$, the choice of the polarization $P$ and the chambers $\fC$ were described in Section 3 of \cite{EllipticHilbert}. Since $X$ is a quiver variety, Conjecture \ref{conj1} holds and thus Theorems \ref{manth} and \ref{thm2} can be applied.

\subsection{} 
Recall  that the Hilbert scheme $X$ is a Nakajima variety associated to the quiver in Fig.\ref{jord}, with dimension $n$, framing dimension $1$ and stability conditions:
$$
\theta_{\pm }: g\to \det(g)^{\pm 1}
$$ 
see \cite{NakajimaLectures1} or Section 3.3 in \cite{EllipticHilbert}.

\begin{figure}[H]
	\centering
	\includegraphics[width=4cm]{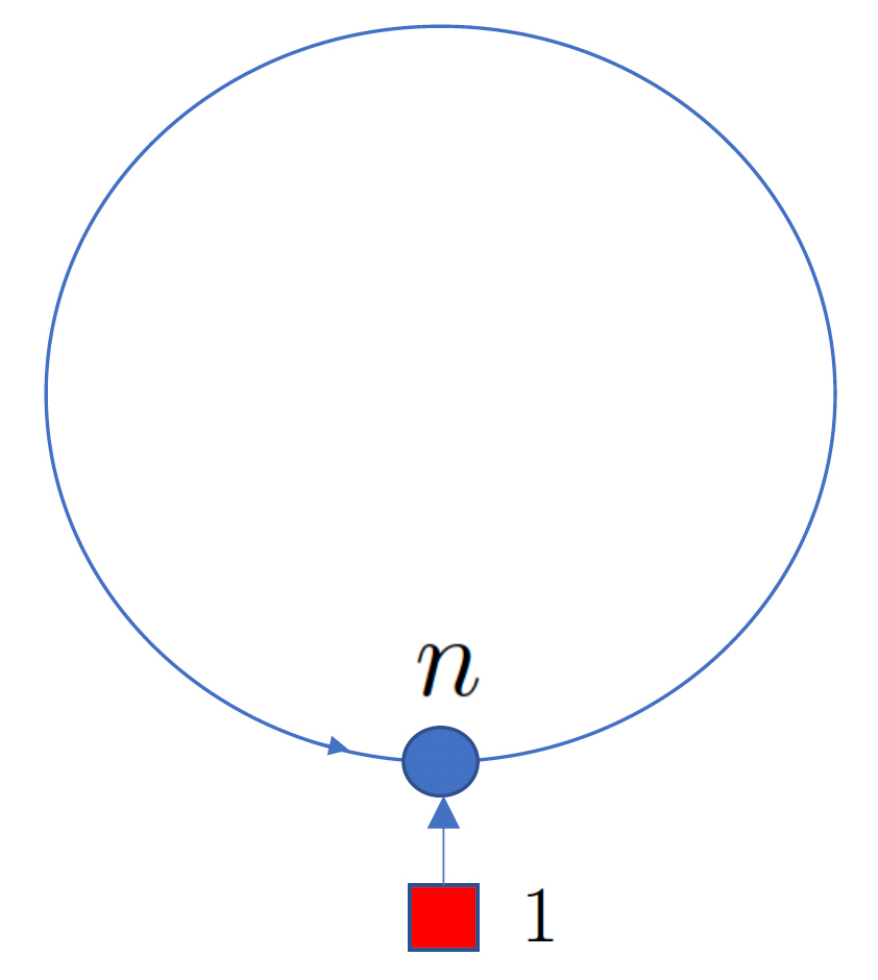}
	\caption{\label{jord} The quiver defining the Hilbert scheme $X$.}
\end{figure}
\noindent
For $b\in\mathbb{N}$ we consider the cyclic subgroup:
$$
\mc= \{ w^{k}, k=0,\dots,b-1 \}\subset \bA\cong \matC^{\times}.
$$
of  $b$-th roots of $1$ generated by $w=e^{2\pi \sqrt{-1}/b}$. 
\begin{Proposition}
{\it The fixed set of $\mc$ has the following form	
$$
X^{\mc}=\coprod_{{n_0,n_1,\cdots, n_{b-1}} \atop {n_0+\dots+n_{b-1}=n} } X(n_0,\dots,n_{b-1})
$$	
where $X(n_0,\dots,n_{b-1})$ is the Nakajima quiver variety associated with 
the cyclic quiver of length $b$ (see Fig.\ref{cyclic}) with dimensions $n_0,\dots,n_{b-1}$,  framing dimensions $r=(1,0,\dots,0)$ and stability conditions
$$
\theta^{b}_{\pm}: (g_0,\dots, g_{b-1})\to \prod\limits_{i=0}^{b-1} \det(g_i)^{\pm 1}.
$$.}
\end{Proposition}
\begin{figure}[H]
	\centering
	\includegraphics[width=4.5cm]{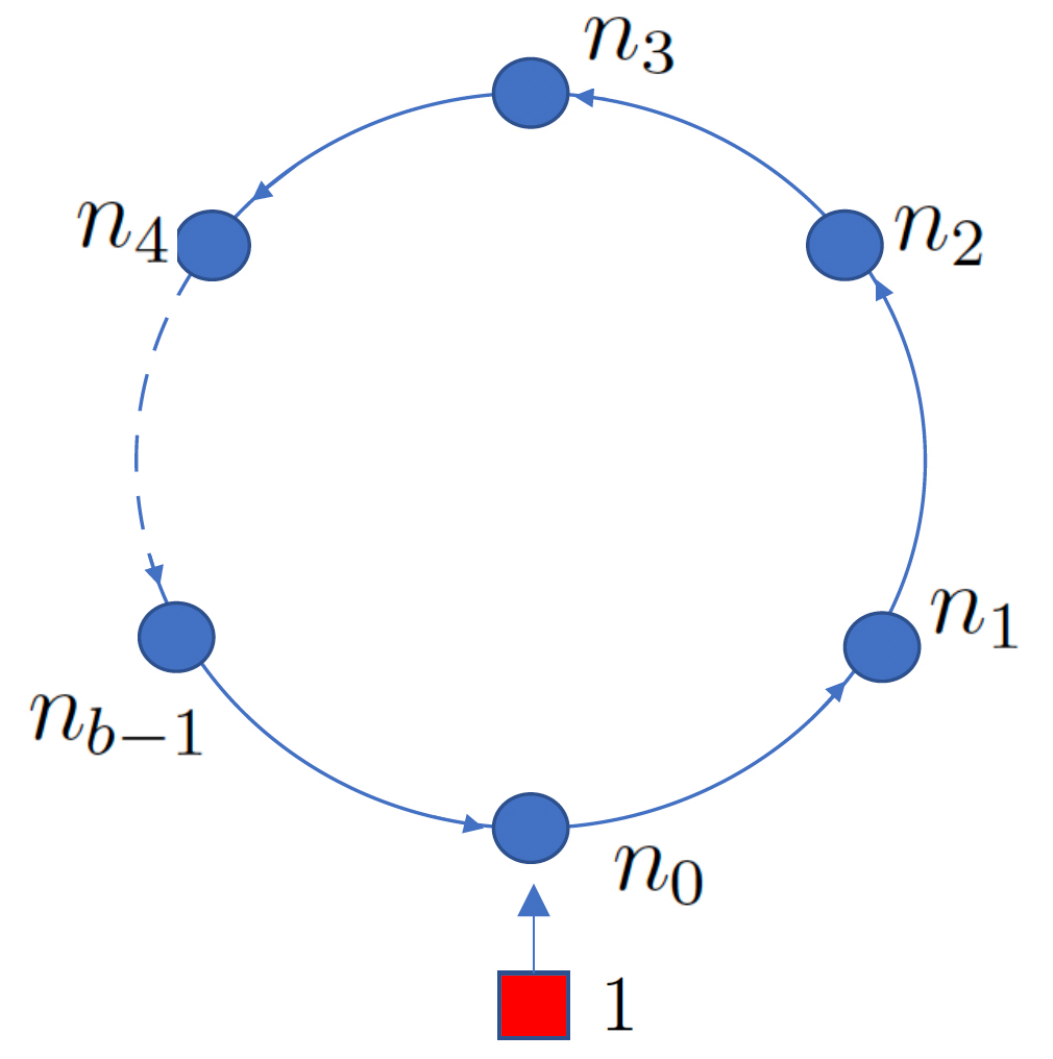}
	\caption{\label{cyclic} The quiver defining $X(n_0,n_1,\dots,n_{b-1})$.}
\end{figure}

We note that it is possible that $X(n_0,\dots,n_{b-1})=\varnothing$ for some choices $(n_0,\dots,n_{b-1})$.

\begin{proof}
Here is the sketch of a proof. As a Nakajima variety associated to Fig.\ref{jord}, $X$ is given by the symplectic reduction of 
$$
T^* R=T^*  \textrm{Hom}(\matC,\matC^n)\oplus T^*   \textrm{Hom}(\matC^n,\matC^n).
$$
by the natural action of $GL(n)$. 
Recall that the torus $\bA$, and thus $\mc$, act by scaling the loop in the quiver.  It means that, if $(J,Y)$ is an element from $R$  then $\mc$ acts by 
$(J,Y)\to (J,Y \omega)$.

We have a decomposition
\be \label{cdecom}
\matC^{n}=\bigoplus_{i=0}^{b-1} \matC^{n_i}, \ \ \  \left.\omega\right|_{\matC^{n_i}}=e^{\frac{2\pi \sqrt{-1} i}{b}}.
\ee
The $\mc$-invariant part of $R$ then has the form
$$
R^{\mc} = \textrm{Hom}(\matC,\matC^{n_0}) \oplus \bigoplus_{i=0}^{b-1} \textrm{Hom}(\matC^{n_i},\matC^{n_i+1}) \ \ \textrm{with}  \ \ \matC^{n_b}:=\matC^{n_0}.
$$
and the symplectic reduction of $T^* R^{\mc}$ is exactly the quiver variety associated to Fig.\ref{cyclic}.  

To complete the proof we also need to show that $\theta^{b}_{\pm}$-stable points in $T^{*} R^{\mc}$ satisfying the moment map condition for $\prod_i GL(n_i)$
are also $\theta_{\pm}$-semistable in $T^{*} R$ and satisfy the moment map condition for $GL(n)$. This is straightforward and we leave it to the reader.
\end{proof}
 
Recall that the fixed points $X^{\bT}=X^{\bA}$ are labeled by the Young diagrams with $n$ boxes. It is also clear from the previous proposition that
$$
X^{\bT}=\coprod\limits_{{n_0,n_1,\cdots, n_{b-1}} \atop {n_0+\dots+n_{b-1}=n} } X(n_0,\dots,n_{b-1})^{\bT}.
$$
\begin{Proposition}
{\it For a fixed point $\lambda \in X^{\bT}$ we have
$$
\lambda\in X(n_0,\dots,n_{b-1}) \ \ \Leftrightarrow \ \  |\{\Box\in \lambda: c(\Box)\!\!\! \mod b=i\}|=n_i, \ \ i=1,\dots,b-1.
$$	
where $c(\Box)$ is the content of a box $\Box$ in the Young diagram $\lambda$.}  
\end{Proposition}
\begin{proof}
It is convenient to use the description of $X$ as a space of ideals in 
$\matC[x,y]$, see Section 3.1-3.2 in \cite{EllipticHilbert}. A box in the Young diagram $\lambda$ with coordinates $(i,j)$ then corresponds to the monomial $x^{j-1} y^{i-1}$. These monomials form a  basis of $\matC^{n}$ in (\ref{cdecom})  above. 
The  $\bA$-character of this monomial equals $i-j~=~c_{\Box}$. 
It means that $\omega$ acts on it by $e^{\frac{2 \pi \sqrt{-1} c_{\Box} }{b}}$ and thus it is from $\matC^{n_{c_{\Box}\!\!\mod b} }$. Since these monomials form a basis, we have
$$
n_{i} =\dim \matC^{n_{i} }=|\{\Box\in \lambda: c(\Box)\!\!\! \mod b=i\}|.
$$
\end{proof}

\subsection{} 
{\rrr The $\bT$-character of the canonical polarization $P_{\lambda}$ is given explicitly by 
\be \label{polarX}
P_{\lambda}=\sum\limits_{i \in \lambda } \varphi^{\lambda}_{i}+\sum\limits_{i,j \in \lambda }\, \dfrac{\varphi^{\lambda}_{i} t_1}{\varphi^{\lambda}_{j}} - \sum\limits_{i,j \in \lambda }\, \dfrac{\varphi^{\lambda}_{i}}{\varphi^{\lambda}_{j}}
\ee
Here the sums are over boxes $i,j$ in the Young diagram $\lambda$ representing the fixed point. The function $\varphi^{\lambda}_{i}$ denotes the $\bT$-content of the box $i$, see Sections 3 of \cite{SmirnovElliptic} for notations. 

To compute the index we substitute $t_1=a\hbar^{1/2}, t_2=a^{-1}\hbar^{1/2}$ to the Laurent polynomial (\ref{polarX}) and collect the terms with positive powers of $a$. Explicitly:
\be \label{indexX}
\ind_{\lambda}:=P_{\lambda,>0}=\sum\limits_{{i \in \lambda}\atop{c_i>0} } \varphi^{\lambda}_{i}+\sum\limits_{{i,j \in \lambda}\atop 
{\atop{c_i-c_j+1>0}}}\, \dfrac{\varphi^{\lambda}_{i} t_1}{\varphi^{\lambda}_{j}} - \sum\limits_{{{i,j \in \lambda }\atop }\atop{c_i-c_j>0}}\, \dfrac{\varphi^{\lambda}_{i}}{\varphi^{\lambda}_{j}}
\ee
Similarity, the $\bA$ -invariant part corresponds to the terms which do not depend on $a$: 
\be \label{stabpart}
P_{\lambda,0}=\sum\limits_{{i \in \lambda}\atop{c_i=0} } \varphi^{\lambda}_{i}+\sum\limits_{{i,j \in \lambda}\atop 
{\atop{c_i-c_j+1=0}}}\, \dfrac{\varphi^{\lambda}_{i} t_1}{\varphi^{\lambda}_{j}} - \sum\limits_{{{i,j \in \lambda }\atop } \atop{c_i-c_j=0}}\, \dfrac{\varphi^{\lambda}_{i}}{\varphi^{\lambda}_{j}}
\ee
The rank of index is the number of terms in (\ref{indexX}) weighted with sign, i.e.:
$$
\textrm{rk}(\ind_{\lambda}):=\sum\limits_{{i \in \lambda}\atop{c_i>0} } 1+\sum\limits_{{i,j \in \lambda} \atop 
{\atop{c_i-c_j+1>0}}}\  1 - \sum\limits_{{i,j \in \lambda } \atop {\atop{c_i-c_j>0}}}\ 1
$$
From (\ref{stabpart}) we also obtain:
$$
\det(P_{\lambda,0})=\Big(\prod\limits_{{i \in \lambda}\atop{c_i=0} } \varphi^{\lambda}_{i}\Big)\Big(\prod\limits_{{i,j \in \lambda}\atop 
{\atop{c_i-c_j+1=0}}}\, \dfrac{\varphi^{\lambda}_{i} t_1}{\varphi^{\lambda}_{j}}\Big) \Big(\prod\limits_{{{i,j \in \lambda }\atop } \atop{c_i-c_j=0}}\, \dfrac{\varphi^{\lambda}_{j}}{\varphi^{\lambda}_{i}}\Big)
$$
Since $\mc$ acts on (\ref{polarX}) by $a\to a\, e^{2 \pi \sqrt{-1}/b}$, the $\mc$-invariant part of the polarization $P^{\mc}_{\lambda}$ is obtained from (\ref{polarX}) by collecting the terms $a^{m}$ with $b\mid m$. Then, $\ind^{\mc}_{\lambda}$ and $\textrm{rk}(\ind^{\mc}_{\lambda})$ are computed from $P^{\mc}_{\lambda}$ in the same way. }

\subsection{}
For the Hilbert scheme $X$ we have  $\bK=\Pic(X)\otimes \matC^{\times}=\matC^{\times}$ and there are two chambers in (\ref{dcham}) corresponding to 
$$
z\to 0 \ \ \textrm{or} \ \  z\to \infty. 
$$
We will denote by $\mathfrak{D}^{\pm}$ the corresponding chambers (\ref{chamdp}) for a $\mc$-fixed point component $X(n_0,\dots,n_b)$.  These chambers correspond to the slopes from {\it canonical} and {\it anticanonical} alcoves of $X(n_0,\dots,n_b)$.

If $s(a,z)$ is as in (\ref{balsec}), then for the Hilbert scheme $X$ it has the following transformation laws:
$$
s(a q,z)=z^{d_{\lambda}-d_{\mu}} s(a,z), \ \  s(a,z q)=a^{d_{\lambda}-d_{\mu}} s(a,z)
$$
and thus $\chi_{\lambda}(\wall,\cdot)=\wall d_{\lambda}$. 
\subsection{} 
Let us choose a $\mc$-fixed component $X(n_{0},\dots,n_{b-1})\subset X$, and consider the matrix: 
$$
\tilde{T}_{\lambda,\mu}(a,z)=\dfrac{\left.\Stab^{Ell}_{X,\fC,P}(\lambda)\right|_{\mu}}{\left.\Stab^{Ell}_{X,\fC,P}(\mu)\right|_{\mu}}, \ \ \ \lambda,\mu \in X(n_0,\dots,n_{b-1})^{\bT}.
$$
and let
$$
\tilde{K}^{\pm}_{\lambda,\mu}(a,\hbar)=\dfrac{\left.\Stab^{\mathfrak{D}^{\pm}}_{X(n_0,\dots,n_{b-1}),\fC,P^{\mc}}(\lambda)\right|_{\mu}}{\left.\Stab^{\mathfrak{D}^{\pm}}_{X(n_0,\dots,n_{b-1}),\fC,P^{\mc}}(\mu)\right|_{\mu}}
$$
be the normalized matrix of restrictions of K-theoretic stable envelopes for the cyclic quiver variety $X(n_0,\dots,n_{b-1})$, with slopes corresponding to the canonical and anticanonical alcoves, then the Theorem 
\ref{thm2} gives:
\begin{Theorem} \label{hsthm}
{\it Let $\wall=\frac{a}{b}\in \matQ$ such that $\mathrm{gcd}(a,b)=1$. Then 
$$
\lim\limits_{z\to 0} Z \Big(\lim_{q\to 0} \tilde{T}(a q^{\wall},z)\Big)  Z^{-1}= H \tilde{K}^{+}(a,\hbar) H^{-1},
$$ 
$$
\lim\limits_{z\to \infty} Z \Big(\lim_{q\to 0} \tilde{T}(a q^{\wall},z)\Big)  Z^{-1}= H \tilde{K}^{-}(a,\hbar) H^{-1},
$$
with diagonal matrices $Z$ and $H$ which have the following diagonal elements:
$$Z_{\lambda,\lambda}=z^{\wall \,d_\lambda}, \ \ H_{\lambda,\lambda}= {\rrr \hbar^{\lfloor \ind_{\lambda} \cdot \wall\rfloor+\mathrm{rk}(\ind_{\lambda})/2
-\mathrm{rk}(\ind^{\mc}_{\lambda})/2} \det(P_{\lambda,0})^{-1/2}.}
$$
}
\end{Theorem}
Let us also denote 
$$
\tilde{T}_{\lambda,\mu}(a,z)=\dfrac{\left.\Stab^{Ell}_{X,\fC,P}(\lambda)\right|_{\mu}}{\left.\Stab^{Ell}_{X,\fC,P}(\mu)\right|_{\mu}}, \ \ \ \lambda,\mu \in X^{\bT}.
$$
We note that $X^{\mc}$ may have nontrivial fixed components (i.e. not just $X^{\bT}$) only if $b\leq n$. The above theorem then gives:
\begin{Corollary}
{\it The limits are non-trivial:	
$$
\lim\limits_{z\to 0} Z \Big(\lim_{q\to 0} \tilde{T}(a q^{\wall},z)\Big)  Z^{-1}\neq \mathrm{Id}
$$
($\mathrm{Id}$ denotes the identity matrix of size $|X^{\bT}|$) only for:
$$
\wall \in \Big\{ \dfrac{a}{b}\in \matQ: \mathrm{gcd}(a,b)=1, \ 1 \leq b\leq n  \Big\}.
$$}	
\end{Corollary}

\section{Finite subgroups of framing torus}
\subsection{}
For this section $X(\nn,\rr)$ denotes a Nakajima quiver variety
with the dimension vector $\nn=(n_1,\dots, n_l)$ and the framing dimensions 
$\rr=(r_1,\dots, r_l)$ where $l$ is the number of vertices in the quiver (see \cite{GinzburgLectures,NakALE} for introductions to quiver varieties). 
The framing torus acting on $X(\nn,\rr)$ has the form
$$
\bA=(\matC^{\times})^{r_1}\times \dots \times (\matC^{\times})^{r_l}.
$$
We denote by $a_1,\dots,a_{|\rr|}$ with $|\rr|=r_1+\dots+r_l$ the coordinates on $\bA$.  We fix the hyperplane arrangement in $\Lie_{\matQ}(\bA)$ defined by the equations:
$$
H^{(n)}_{i,j}=\{ \tilde{a}_i-\tilde{a}_j=n \} \subset \Lie_{\matQ}(\bA), \ \ i,j \in I, \ \ n\in \matZ,
$$ 
where $\tilde{a}_i$, $i \in I=\{1,\dots,|\rr|\}$ denote the corresponding coordinates on~$\Lie_{\matQ}(\bA)$. 

{\rrr For a subset of indices $J \subset I$ we associate a one-dimensional subtorus
$$
\iota_{J}: \matC^{\times} \hookrightarrow \bA
$$
defined by $a_i=z$ if $i\in J$ or $a_i=1$ otherwise, where $z$ denotes the cordinate on $\mathbb{C}^{\times}$. }

\subsection{}
{\rrr Let us fix a point $\wall \in \Lie_{\matQ}(\bA)$. We say that two indices $i,j \in I$ are equivalent if $\wall\in H^{(n)}_{i,j}$  for some $n$. This equivalence relation induces a decomposition of the set of indices: 
\be \label{isubs}
I=I_{1}\cup \dots \cup I_{m}
\ee
Note that
\be \label{rdec}
\rr=\rr_1+\dots+\rr_m
\ee
with $|\rr_k|=|I_k|$. 
Let $\bZ(\wall) \subset \bA$ be the $m$-dimensional subtorus 
given by 
\be \label{torincl}
\iota_{I_1}\times \dots \times \iota_{I_m}: (\mathbb{C}^{\times})^{m} \hookrightarrow \bA. 
\ee
For example, if $\wall=0$ then all indexes are equivalent to each other, i.e., $m=1$ and (\ref{isubs}) takes the form
$
I=I_1.
$
Thus $\bZ(\wall)\cong\mathbb{C}^{\times}$. By definition (\ref{torincl}) this subtorus acts by scaling all framing spaces with the same weight, thus it acts trivially on $X$.

The other extreme case is when $\wall$ is generic, i.e., does not belong to any of the walls $H^{(n)}_{i,j}$. In this case $m=|\rr|$ and each subset in (\ref{isubs}) contains only one index.  In this case $\bZ(\wall)=\bA$.}

We recall the following well known property of quiver varieties (the tensor product structure):
\begin{Lemma} \label{framlem}
{\it The fixed point set of the torus $\bZ(\wall)$ has the following form:
$$
X(\nn,\rr)^{\bZ(\wall)}=\coprod_{\nn_1+\dots+\nn_m=\nn} \, X(\nn_1,\rr_1)\times \dots \times X(\nn_m,\rr_m).
$$	
The $\bA$-weights appearing in the normal bundle $N_{X(\nn,\rr)^{\bZ(\wall)}}$ are  of the form $a_i/a_j$ with
$i$ and $j$ from different subsets of decomposition (\ref{isubs}). }
\end{Lemma}
\begin{proof}
	See Section 2.4 in \cite{MO}.
\end{proof}

For example, if $\wall=0$ then  
$X(\nn,\rr)^{\bZ(\wall)}=X(\nn,\rr)$. For generic $\wall$ $X(\nn,\rr)^{\bZ(\wall)}=X(\nn,\rr)^{\bA}$. 

Informally speaking, we have the following picture. For each point $\wall\in \Lie_{\matQ}(\bA)$ we associate a subvariety $X(\nn,\rr)^{\bZ(\wall)}$ in $X(\nn,\rr)$.
For a  point $\wall$ which is in the complement of all hyperplanes, this subvariety is simply $X(\nn,\rr)^{\bA}$. If $\wall$ arrives at a hyperplane then the subvariety gets larger. Further, if $\wall$ is at an intersection of two hyperplanes the fixed point set gets even larger and so on. Finally, when we arrive at the intersection of maximal number of hyperplanes the corresponding variety gets maximally  large, i.e., $X(\nn,\rr)$.

\subsection{}

Let $\fC$ and $P$ be a choice of a chamber and a polarization for a  quiver variety $X(\nn,\rr)$. We denote the $\bZ(\wall)$-invariant part of $P$ by $P(\wall)$. Clearly,
$$
P(\wall)=\bigoplus_{i=0}^{m} P_{i}
$$  
where $P_i$ is a polarization for $X(\nn_i,\rr_i)$. We denote by 
$\ind^{\wall}_{\lambda}$ the index of $\lambda$ associated with $P(\wall)$ and the chamber $\fC$.

As the varieties $X(\nn,\rr)$ and $X(\nn_i,\rr_i)$ are all associated to the same quiver, the map $\kappa$ is an isomorphism and we write $\mathfrak{D}'=\mathfrak{D}$.

For a point $\wall\in \Lie_{\matQ}(\bA)$, as in the previous section,
we denote by $\omega_{\wall}$ the translation acting on sections of line bundles over $\cE_{\bA}$ by:
\be \label{wsh}
\omega^{*}_{\wall}f(a_1,\dots, a_{|\rr|})=f(a_1 q^{\wall_1},\dots, a _{|\rr|} q^{\wall_{|\rr|}}). 
\ee
Theorem \ref{manth} then gives:

\begin{Theorem} \label{manth2}
	{\it For any $\wall\in \Lie_{\matQ}(\bA)$ we have:
		\begin{align} \label{mainlim2}
		&\Lambda^{\!\bullet}(\bar{P}(\wall)) \circ \lim\limits_{z\to 0_{\mathfrak{D}}} \left( z^{\chi(\wall,\cdot)-\chi_{\lambda}(\wall,\cdot)}  \lim\limits_{q\to 0}  \omega^{*}_{\wall} \circ i^{*} \left( \dfrac{\Stab^{Ell}_{X(\nn,\rr),\fC,P} (\lambda) }{\Theta(P)}  \right) \right) \circ \det(P_{0,\lambda})^{1/2} \\ \nonumber
		& =\hbar^{\lfloor \ind_{\lambda}\cdot \wall\rfloor +\mathrm{rk}(\ind_{\lambda})/2- \mathrm{rk}(\ind^{\mc}_{\lambda})/2}\, \Stab^{\mathfrak{D}}_{X(\nn,\rr)^{\bZ(\wall)},\fC,P(\wall)}(\lambda),
		\end{align}
		where $X(\nn,\rr)^{\bZ(\wall)}$ is described in Lemma \ref{framlem}.
		}
\end{Theorem}
\begin{proof}
{\rrr We have $\wall=(\wall_{1},\dots,\wall_{|r|})$ with $\wall_i\in \matQ$. We consider a cyclic subgroup of $\mc \subset \bA$ generated by the element:
\be \label{nudef}
(e^{2\pi \sqrt{-1} \wall}): = (e^{2\pi \sqrt{-1} \wall_1},\dots,  e^{2\pi \sqrt{-1} \wall_{|\rr|}}) \in \bA.
\ee
Clearly  $\mc \subset \bZ(\wall)$ and thus we have an $\bA$-equivariant embedding $X(\nn,\rr)^{\bZ(\wall)}\subset X(\nn,\rr)^{\mc}$.

If  $X(\nn,\rr)^{\bZ(\wall)}\neq X(\nn,\rr)^{\mc}$ then we have a non-zero rank, $\bA$-equivariant normal bundle   $N$ to $X(\nn,\rr)^{\bZ(\wall)}$ in $X(\nn,\rr)^{\mc}$. By Lemma \ref{framlem} all the $\bA$ - weights appearing in $N$ are of the form $a_i/a_j$ with indices $i$ and $j$ from different subsets in decomposition (\ref{isubs}). The subgroup $\mc$ acts on such weights via
$$
a_i/a_j \to a_i/a_j e^{2 \pi \sqrt{-1} (\wall_i-\wall_j)}.
$$
Since this action must be trivial we have $\wall_i-\wall_j=n\in \mathbb{Z}$, which means that $\wall\in H^{(n)}_{i,j}$. Thus $i$ and $j$ must be from the same subset in decomposition (\ref{isubs}).  We arrive at a contradiction,  thus $X(\nn,\rr)^{\bZ(\wall)}= X(\nn,\rr)^{\mc}$

Finally, we see that  the shift (\ref{wsh}) satisfies the conditions described in Section \ref{wshift} for $\mc$, 
 and the result follows from Theorem \ref{manth}.  }
\end{proof}

\section*{Conflict of interests} 
On behalf of all authors, the corresponding author states that there is no conflict of interest.

\bibliographystyle{abbrv}
\bibliography{bib}

\newpage

\vspace{12 mm}

\noindent
Yakov Kononov\\
Department of Mathematics,\\
Columbia University,\\
New York, NY 10027, USA\\
{\it and}\\
Department of Mathematics,\\
Yale University,\\
New Haven, CT 06511, USA
ya.kononoff@gmail.com

\vspace{3 mm}

\noindent
Andrey Smirnov\\
Department of Mathematics,\\
University of North Carolina at Chapel Hill,\\
Chapel Hill, NC 27599-3250, USA\\
{\it and}\\
Steklov Mathematical Institute \\
of Russian Academy of Sciences, \\
Gubkina str. 8, Moscow, 119991, Russia \\
asmirnov@email.unc.edu

\end{document}